\documentclass[12pt, twoside]{article}
\usepackage{amsmath,amsthm,amssymb}
\usepackage[utf8]{inputenc}
\usepackage{times}
\usepackage{enumerate}
\usepackage{needspace}
\usepackage{tikz}
\usetikzlibrary{calc,positioning,shadows.blur,decorations.pathreplacing,patterns}
\usepackage{wrapfig}

\newcommand{\txtand}{{\mathop{\text{~and~}}}}

\newcommand{\Acal}{{\mathcal{A}}}
\newcommand{\Bcal}{{\mathcal{B}}}
\newcommand{\Ccal}{{\mathcal{C}}}

\newcommand{\Fcal}{{\mathcal{F}}}

\newcommand{\Ical}{{\mathcal{I}}}

\newcommand{\Ocal}{{\mathcal{O}}}

\newcommand{\contract}{.}

\newcommand{\BS}{\backslash}

\newcommand{\rk}{\mathrm{rk}}
\newcommand{\cl}{\mathrm{cl}}

\newcommand{\Z}{\mathbb{Z}}

\newcommand{\N}{\mathbb{N}}

\renewcommand{\phi}{\varphi}

\newcommand{\maparrow}{\longrightarrow}

\newcommand{\BSET}[1]{{\BS\left\{ #1 \right\}}}
\newcommand{\SET}[1]{{\left\{ #1 \right\}}}

\newcommand{\COMMENT}[1]{}
\newcommand{\XXXCUTXXX}[1]{}

\newcommand{\ROMANENUM}{\renewcommand{\theenumi}{(\roman{enumi})}\renewcommand{\labelenumi}{\theenumi}}

\def\titlerunning#1{\gdef\titrun{#1}}
\makeatletter
\def\author#1{\gdef\autrun{\def\and{\unskip, }#1}\gdef\@author{#1}}
\def\address#1{{\def\and{\\\hspace*{18pt}}\renewcommand{\thefootnote}{}%
\footnote {#1}}%
\markboth{\autrun}{\titrun}}
\makeatother
\def\email#1{e-mail: #1}
\def\subjclass#1{{\renewcommand{\thefootnote}{}%
\footnote{\emph{Mathematics Subject Classification (2010):} #1}}}
\def\keywords#1{\par\medskip
\noindent\textbf{Keywords.} #1}

\newtheorem{theorem}{Theorem}[section]
\newtheorem{corollary}[theorem]{Corollary}
\newtheorem{proposition}[theorem]{Proposition}
\newtheorem{lemma}[theorem]{Lemma}

\theoremstyle{definition}
\newtheorem{definition}[theorem]{Definition}
\newtheorem{remark}[theorem]{Remark}
\newtheorem{example}[theorem]{Example}

\begin{document}


\titlerunning{Lattice Path Matroids are 3-Colorable}

\title{Lattice Path Matroids are 3-Colorable}

\author{Immanuel Albrecht
\and 
Winfried Hochstättler}

\date{\today}

\maketitle

\address{I.~Albrecht: 
FernUniversität in Hagen, Fakultät für Mathematik und Informatik, Lehrgebiet für Diskrete Mathematik und Optimierung, D-58084~Hagen, Germany;
\email{Immanuel.Albrecht@fernuni-hagen.de}
\and
W.~Hochstättler: FernUniversität in Hagen, Fakultät für Mathematik und Informatik, Lehrgebiet für Diskrete Mathematik und Optimierung, D-58084~Hagen, Germany; \email{Winfried.Hochstaettler@fernuni-hagen.de}}



\begin{abstract}
We show that every lattice path matroid of rank at least two has a quite simple coline,
also known as a positive
coline. Therefore every orientation of a lattice path matroid is $3$-colorable with respect to the chromatic number
of oriented matroids introduced by J.~Nešetřil, R.~Nickel, and W.~Hochstättler.

\end{abstract}


Recently, in order to verify the generalization of Hadwiger's
Conjecture to oriented matroids for the case of 3-colorability, Goddyn
et.\ al.\ \cite{GoHoNe15} introduced the class of generalized series
parallel (GSP) matroids and asked whether it coincides with the class
of oriented matroids without $M(K_4)$-minor. Furthermore, they showed
that a minor closed class $\mathcal{C}$ of oriented matroids is a
subclass of the GSP-matroids, if every simple matroid in $\mathcal{C}$
contains a flat of codimension 2, i.\ e.\ a coline, which is contained
in more flats of codimension 1, i.\ e.\ copoints, with only one extra
element, than in larger copoints. We call such a coline quite simple.
They conjectured that every simple gammoid of rank at least 2
has a quite simple coline. Gammoids may be characterized as the smallest class of matroids
that is closed under minors and under duality, and which contains all
transversal matroids -- a class of matroids that is not closed under minors nor duals.
Bicircular matroids
form a minor closed subclass of the transversal matroids, and Goddyn et.\
al.\ \cite{GoHoNe15} verified the existence of a quite simple coline in
every simple bicircular matroid of rank at least 2.

Another minor closed subclass of the transversal matroids is the class of the lattice path matroids
\cite{Bonin2006701}. In this work we show that every simple lattice path matroids of rank at least 2 has a quite simple coline,
which implies that orientations of lattice path matroids are GSP, and therefore we obtain
the 3-colorability of every orientation of a lattice path matroid.

\section{Preliminaries}

In this work, we consider \emph{matroids} to be pairs $M=(E,\Ical)$ where $E$ is a finite set and $\Ical$ is a system of
independent subsets of $E$ subject to the usual axioms (\cite{Ox11}, Sec.~1.1). Furthermore, \emph{oriented matroids} are
considered triples $\Ocal = (E,\Ccal,\Ccal^\ast)$ where $E$ is a finite set, $\Ccal$ is a family of signed circuits and $\Ccal^\ast$
is a family of signed cocircuits subject to the axioms of oriented matroids (\cite{BlVSWZ99}, Ch.~3). Every oriented matroid $\Ocal$
has a uniquely determined underlying matroid defined on the ground set $E$, which we shall denote by $M(\Ocal)$.\footnote{The underlying matroid is the only notion from oriented matroids that is needed for the comprehension of this work.}

\begin{definition}[\cite{GoHoNe15}, Definition~4]
 Let $M=(E,\Ical)$ be
a matroid. A flat $X\in\Fcal(M)$ is called \emph{coline of $ M$},
 if $\rk_M(X)=\rk_M(E)-2$.
 A flat $Y\in\Fcal(M)$ is called
\emph{copoint of $ M$ on $ X$},
 if $X\subseteq Y$ and
$\rk_M(Y)=\rk_M(E)-1$. 
If further $\left| Y\backslash X \right|=1$,
we say that $Y$ is a \emph{simple copoint on $ X$}. 
If otherwise $\left| Y\backslash X \right|>1$, we
say that $Y$ is a \emph{multiple copoint on $ X$}\footnote{In \cite{GoHoNe15} multiple copoints are called {\em fat copoints}.\index{fat copoint}}.
 A \emph{quite simple coline}\footnote{In \cite{GoHoNe15} quite simple colines are called {\em positive colines}.\index{positive coline}} is a coline $X\in \Fcal(M)$,
 such that there are more simple copoints on $X$ than there are multiple copoints on $X$.
\end{definition}

 The following definitions are basically those found in J.E.~Bonin and A.~deMier's paper {\em Lattice path matroids: Structural properties} \cite{Bonin2006701}.
\begin{definition}
Let $n\in\mathbb{N}$. A \emph{lattice path} of length $n$ is a tuple\label{n:latticePath}
$(p_{i})_{i=1}^{n}\in\{\mathrm{N},\mathrm{E}\}^{n}$. 
We say
that the \emph{$ i$-th step} of $(p_{i})_{i=1}^{n}$ is towards the North if $p_{i}=\mathrm{N}$,
and towards the East if $p_{i}=\mathrm{E}$.
\end{definition}

\begin{definition}
Let $n\in\mathbb{N}$, and let $p = (p_{i})_{i=1}^{n}$ and $q = (q_{i})_{i=1}^{n}$
be lattice paths of length $n$. We say that $p$
is \emph{south of $ q$}
if for all $k\in\SET{1,2,\ldots,n}$, 
\[
\left|\left\{ i\in \N \BSET {0} \vphantom{A^A}~\middle|~ i\leq k \txtand p_{i}=\mathrm{N}\right\} \right|\leq\left|\left\{ i\in \N \BSET {0} \vphantom{A^A}~\middle|~ i\leq k \txtand q_{i}=\mathrm{N}\right\} \right|.
\]
We say that $p$ and $q$ have \emph{common endpoints}, if
 $$\left|\left\{ i\in \N \BSET {0} \vphantom{A^A}~\middle|~ i\leq n \txtand p_{i}=\mathrm{N}\right\} \right| = \left|\left\{ i\in \N \BSET {0} \vphantom{A^A}~\middle|~ i\leq n \txtand q_{i}=\mathrm{N}\right\} \right|$$ holds. 
 We say that the \emph{lattice path $ p$
is south of  $ q$ with common endpoints},
 if $p$ and $q$ have common endpoints and $p$ is south of $q$.
 In this case, we write $p \preceq q$.\label{n:neverabove}
\end{definition}

\begin{definition}
Let $n\in \N$, and  let $p,q\in \SET{\mathrm{E},\mathrm{N}}^n$ be lattice paths such
that $p\preceq q$. We define the set
of \emph{lattice paths between
 $ p$ and $ q$}
to be\label{n:LPbetweenPQ}
\[
\mathrm{P}\left[p,q\right]=\left\{ r \in\{\mathrm{N},\mathrm{E}\}^{n}
\vphantom{A^A}~\middle|~
p \preceq r \preceq q\right\} . \qedhere
\]
\end{definition}

\begin{definition}\label{def:LPmatroid}
A matroid $M=(E,\Ical)$ is called \emph{strong lattice path matroid}, if
its ground set has the property 
$E = \SET{1,2,\ldots,\left| E \right|}$ and if
there are lattice paths $p,q\in \SET{\mathrm{E},\mathrm{N}}^{\left| E \right|}$
with $p\preceq q$,
such that $M = M[p,q],$
where $M[p,q]$ denotes the transversal matroid  presented by
the family
$\Acal_{[p,q]}=(A_i)_{i=1}^{\rk_M(E)} \subseteq E$
with
\[
A_i = \SET{j\in E~\middle|~\exists (r_j)_{j=1}^{\left| E \right|}\in \mathrm{P}[p,q]\colon\,
 r_j = \mathrm{N}\txtand \left| \SET{k\in E\mid k\leq j,\,r_k=\mathrm{N} }\right| = i },
\]
i.e. each $A_i$ consists of those $j\in E$, such that there is a 
lattice path $r$ between $p$ and $q$ such that the $j$-th step of $r$ is towards the North for the $i$-th time in total.
Furthermore,
a matroid $M=(E,\Ical)$ is called \emph{lattice path matroid}, 
if there is a bijection $\phi \colon E\maparrow \SET{1,2,\ldots,\left| E \right|}$ 
such that $\phi[M] = \left(\phi[E],\SET{\phi[X]\vphantom{A^A}~\middle|~ X\in\Ical}\right)$ is a strong lattice path matroid.
\end{definition}


\begin{example}\label{ex:A}
(Fig.~\ref{fig:A}a) 
Let us consider the two lattice paths $p=(\mathrm{E},\mathrm{E},\mathrm{N},\mathrm{E},\mathrm{N},\mathrm{N})$
and $q=(\mathrm{N},\mathrm{N},\mathrm{E},\mathrm{N},\mathrm{E},\mathrm{E})$. 
We have $p\preceq q$ and the strong lattice path matroid $M[p,q]$ is the transversal matroid $M(\Acal)$ presented
by the family  $\Acal=(A_i)_{i=1}^3$ of subsets of $\SET{1,2,\ldots,6}$ where $A_{1}=\{1,2,3\}$, $A_{2}=\{2,3,4,5\}$,
and 
$A_{3}=\{4,5,6\}$.%
\end{example}

\begin{theorem}[\cite{Bonin2006701}, Theorem~2.1]
\label{LPMthm:P} Let $p$, $q$ be lattice
paths of length $n$, such that $p\preceq q$.
Let $\Bcal \subseteq 2^{\SET{1,2,\ldots,n}}$ consist of the bases of the
strong lattice path matroid $M=M[p,q]$ on the ground set $E=\SET{1,2,\ldots,n}$.
Let
$$\phi \colon \mathrm{P}[p,q] \maparrow \Bcal,\quad (r_i)_{i=1}^n \mapsto \SET{j\in \N ~\middle|~ 1\leq j\leq n,\,r_j = \mathrm{N}} .$$
Then $\phi$ is a bijection between 
the family of lattice paths $\mathrm{P}[p,q]$ between $p$ and $q$ and the family of bases of $M$.
\end{theorem}
\begin{proof}
	Clearly, $\phi$ is well-defined: let $r=(r_i)_{i=1}^n\in \mathrm{P}[p,q]$, and let $m = \rk_M(E)$,
	then there are $j_1 < j_2 < \ldots < j_m$ such that $r_i = \mathrm{N}$ if and only if $i\in \SET{j_1,j_2,\ldots,j_m}$.
	Thus the map $$\iota_r\colon \phi(r) \maparrow \SET{1,2,\ldots,m},$$ where
	$\iota_r(i) = k$ for $k$ such that $i = j_k$, witnesses that the set $\phi(r) \subseteq \SET{1,2,\ldots,n}$ is indeed a transversal of $\Acal_{[p,q]}$,
	and therefore a base of $M[p,q]$. It is clear from Definition~\ref{def:LPmatroid} that $\phi$ is surjective.
	It is obvious that if we consider only lattice paths of a fixed given length $n$, then the indexes of the steps towards the North
	 uniquely determine such a lattice path. Thus $\phi$ is also injective.
\end{proof}

\begin{theorem}[\cite{Bonin2006701}, Theorem~3.1]\label{thm:LPMclosedUnderStuff}
The class of lattice path matroids
is closed under minors, duals and direct sums.\end{theorem}

\section{The Western Coline}

\begin{figure}[h]
\begin{center}
\begin{tabular}{lll}
\textit{a)}&\textit{b)}&\textit{c)}\\
\begin{tikzpicture}[scale=.85]
\node[inner sep=0pt,circle,minimum size=4pt,fill] at (0,0) {};

\node[inner sep=0pt,circle,minimum size=4pt,fill] at (1,0) {};

\node[inner sep=0pt,circle,minimum size=4pt,fill] at (2,0) {};

\node[inner sep=0pt,circle,minimum size=4pt,fill] at (0,1) {};

\node[inner sep=0pt,circle,minimum size=4pt,fill] at (1,1) {};

\node[inner sep=0pt,circle,minimum size=4pt,fill] at (2,1) {};

\node[inner sep=0pt,circle,minimum size=4pt,fill] at (3,1) {};
\node[inner sep=0pt,circle,minimum size=4pt,fill] at (0,2) {};

\node[inner sep=0pt,circle,minimum size=4pt,fill] at (1,2) {};

\node[inner sep=0pt,circle,minimum size=4pt,fill] at (2,2) {};

\node[inner sep=0pt,circle,minimum size=4pt,fill] at (3,2) {};
\node[inner sep=0pt,circle,minimum size=4pt,fill] at (1,3) {};

\node[inner sep=0pt,circle,minimum size=4pt,fill] at (2,3) {};

\node[inner sep=0pt,circle,minimum size=4pt,fill] at (3,3) {};

\draw[thin] (1,0) -- (1,3);
\draw[thin] (2,0) -- (2,3);
\draw[thin] (0,1) -- (3,1);
\draw[thin] (0,2) -- (3,2);
\draw[very thick] (0,0) -- (2,0) -- (2,1) -- (3,1) -- (3,3);
\node at (3.5,.5) {$p$};
\node at (-.5,.5) {$q$};

\draw[very thick] (0,0) -- (0,2) -- (1,2) -- (1,3) -- (3,3);
\begin{scope}[shift={(-.1,-.15)}]
\node at (.3,.5) {$1$};
\node at (1.3,.5) {$2$};
\node at (2.3,.5) {$3$};
\node at (.3,1.5) {$2$};
\node at (1.3,1.5) {$3$};
\node at (2.3,1.5) {$4$};
\node at (3.3,1.5) {$5$};
\node at (1.3,2.5) {$4$};
\node at (2.3,2.5) {$5$};
\node at (3.3,2.5) {$6$};
\end{scope}
\end{tikzpicture}
& \includegraphics[scale=.75]{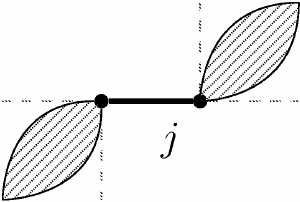} & \includegraphics[scale=.75]{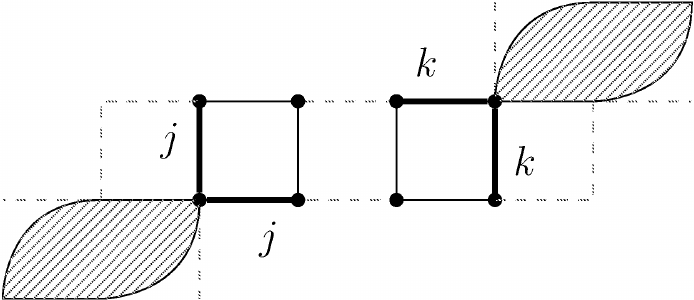}
\end{tabular}
\end{center}
\caption{\label{fig:A}\textit{a)} Lattice paths for Ex.~\ref{ex:A}, \textit{b,c)} situation in Prop.\ref{LPMprop:X} \textit{(ii)} and \textit{(iii)}.}
\end{figure}
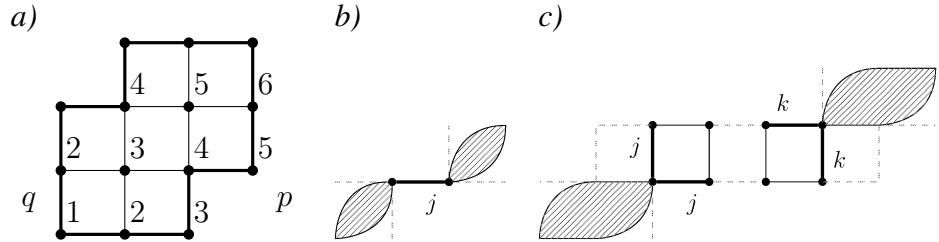

\begin{proposition}
\label{LPMprop:X}Let $p=(p_{i})_{i=1}^{n}$, $q=(q_{i})_{i=1}^{n}$
be lattice paths of length $n$ such that $p\preceq q$. Let $j\in E=\SET{1,2,\ldots,n}$ and $M=M[p,q]$. Then
\begin{enumerate} \ROMANENUM
\item $\rk_M\left( \SET{1,2,\ldots,j} \right)=\left|\left\{ i\in \SET{1,2,\ldots,j} \mid q_{i}=\mathrm{N}\right\} \right|$.
\item 
The element $j$ is a loop in $M$ if and only if 
\[
\left|\left\{ i\in \SET{1,2,\ldots,j-1} \vphantom{A^A}~\middle|~ p_{i}=\mathrm{N}\right\} \right|=\left|\left\{ i\in \SET{1,2,\ldots,j}\vphantom{A^A}~\middle|~ q_{i}=\mathrm{N}\right\} \right|,
\]
i.e. the $j$-th step is forced to go towards East for all $r\in \mathrm{P}[p,q]$ (Fig.~\ref{fig:A}b).
\item For all $k\in E$ with $j<k$, $j$ and $k$ are parallel edges in $M$
if and only if 
\begin{eqnarray*}
\left|\left\{ i\in\SET{1,2,\ldots,j-1}\vphantom{A^A}~\middle|~p_{i} = \mathrm{N}\right\} \right| & = & \left|\left\{ i\in\SET{1,2,\ldots,k-1}\vphantom{A^A}~\middle|~ p_{i}=\mathrm{N}\right\} \right|
\\ & =&
\left|\left\{ i\in\SET{1,2,\ldots,j}\vphantom{A^A}~\middle|~q_{i} = \mathrm{N}\right\} \right| -1
\\& = & \left|\left\{ i\in\SET{1,2,\ldots,k}\vphantom{A^A}~\middle|~q_{i}=\mathrm{N}\right\} \right| -1, \end{eqnarray*}
i.e. the $j$-th and $k$-th steps of any $r\in \mathrm{P}[p,q]$ are in a common corridor towards
the East that is one step wide towards the North (Fig.~\ref{fig:A}c).
\end{enumerate}
\end{proposition}
\begin{proof}
For every $r\in \mathrm{P}[p,q]$, we have $r \preceq q$,
therefore $r$ is south of $q$, thus for
 all $k\in E$,
$\left| \SET{j\in \SET{1,2,\ldots,k} \vphantom{A^A}~\middle|~r_k = \mathrm{N}} \right| \leq \left| \SET{j\in \SET{1,2,\ldots,k}\vphantom{A^A}~\middle|~ q_k = \mathrm{N}} \right|$.
Consequently, 
$\left\{ i\in \SET{1,2,\ldots,j} \vphantom{A^A}~\middle|~ q_{i}=\mathrm{N}\right\}$ is a maximal independent subset of $\SET{1,2,\ldots,j}$
and so statement {\em (i)} holds.
%
An element  $j\in E$ is a loop in $M$, if and only if $\mathrm{rk}_M(\{j\})=0$,
which is the case if and only $\SET{j}$ is not independent in $M$.
This is the case if and only if for all bases $B$ of $M$, $j\notin B$ holds, 
because every independent set is a subset of a base.
The latter holds
 if and only if for all $(r_{i})_{i=1}^{n}\in\mathrm{P}[p,q]$ the $j$-th step is towards the East, i.e. $r_{j}=\mathrm{E}$.
 This, in turn, is the case
if and only if $\left|\left\{ i\in\SET{1,2,\ldots,j-1}\vphantom{A^A}~\middle|~p_{i}=\mathrm{N}\right\} \right|=\left|\left\{ i\in\SET{1,2,\ldots,j}\vphantom{A^A}~\middle|~q_{i}=\mathrm{N}\right\} \right|$.
Thus statement {\em (ii)} holds, too.
Let $j,k\in E$ with $j< k$. It is easy to see that if $j$ and $k$ are in a common corridor, then every lattice path $r=(r_i)_{i=1}^{n}$ of length $n$
with $r_j = r_k = \mathrm{N}$ cannot be between $p$ and $q$, i.e. $p\preceq r \preceq q$ cannot hold: a lattice path $r$ with $r_j = r_k = \mathrm{N}$
is either below $p$ at $j-1$ or above $q$ at $k$. 
Thus $\SET{j,k}$ cannot be independent in $M$. By {\em (i)}, neither $j$ nor $k$ can be a loop in $M$, thus $j$ and $k$ must be parallel edges in $M$.
Conversely, let $j < k$ be parallel edges in $M$. 
Then $j$ is not a loop in $M$, so there is a path $r^1 = (r_{i}^{1})_{i=1}^{n}\in\mathrm{P}[p,q]$
with $r_{j}^{1}=\mathrm{N}$ which is minimal with regard to $\preceq$,
and then
\[
\left|\left\{ i\in \SET{1,2,\ldots,j-1}\vphantom{A^A}~\middle|~ r_{i}^{1}=\mathrm{N}\right\} \right|=\left|\left\{ i\in\SET{1,2,\ldots,j-1}\vphantom{A^A}~\middle|~p_{i}=\mathrm{N}\right\} \right|.
\]
Since $j$ and $k$ are parallel edges, $\SET{j,k}\not\subseteq B$ for all bases $B$ of $M$.
Therefore there is no $r=(r_{i})_{i=1}^{n}\in \mathrm{P}[p,q]$ such that $r_i = r_k = \mathrm{N}$.
This yields the equation
\begin{align*}
\left|\left\{ i\in\SET{1,2,\ldots,k}\vphantom{A^A}~\middle|~q_{i} = \mathrm{N}\right\} \right|  \,\,=\,\, & \left| \left\{ i\in\SET{1,2,\ldots,j}\vphantom{A^A}~\middle|~ r_{i}^{1}=\mathrm{N}\right\}\right|
\\   = \,\,& \left|\left\{ i\in\SET{1,2,\ldots,j-1}\mid r_{i}^{1}=\mathrm{N}\right\} \right|+1.
\end{align*}
Since $k$ is not a loop in $M$, it follows that \[
\left|\left\{ i\in\SET{1,2,\ldots,j-1}\vphantom{A^A}~\middle|~ p_{i}=\mathrm{N}\right\} \right|=\left|\left\{ i\in\SET{1,2,\ldots,j}\vphantom{A^A}~\middle|~ q_{i}=\mathrm{N}\right\} \right|-1.
\] Thus {\em (iii)} holds.
\end{proof}

\begin{lemma}
\label{LPMlem:A}Let $p=(p_{i})_{i=1}^{n}$ and $q=(q_{i})_{i=1}^{n}$ be
lattice paths of length $n$, such that $p\preceq q$, and such that $M=M[p,q]$ is
a strong lattice path matroid on $E=\SET{1,2,\ldots,n}$ which has no loops.
Let $j\in E$ such that $q_j = \mathrm{N}$. Then 
\[
\SET{1,2,\ldots,j-1}=\cl_M\left(\SET{1,2,\ldots,j-1}\right).
\]
Furthermore, for all $k\in E$ with $k \geq j$, 
\[
\rk_M\left( \SET{1,2,\ldots,j-1}\cup\{k\} \right)= \rk_M\left( \SET{1,2,\ldots,j-1} \right)+1.
\]
 \end{lemma}
 

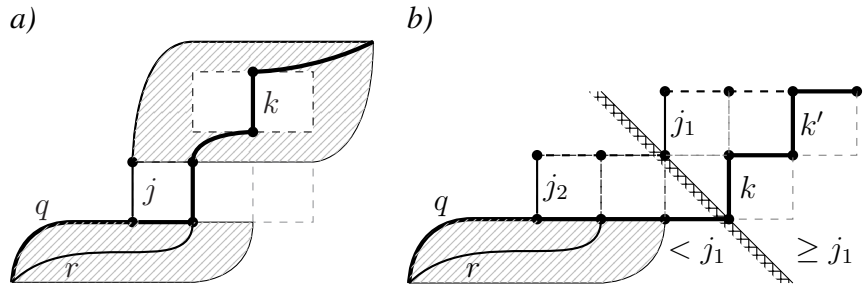
\begin{figure}[b]
\caption{The lattice paths $s$ in the proof of  \textit{a)} Lem.~\ref{LPMlem:A} and \textit{b)} Thm.~\ref{thm:WesternColine}.}
\begin{center}
\begin{tabular}{ll}
\textit{a)} & \textit{b)} \\
\begin{tikzpicture}[scale=0.8]\selectcolormodel{cmyk}
\node[inner sep=0pt,circle,minimum size=4pt,fill] (j) at (0,0) {};
\node[inner sep=0pt,circle,minimum size=4pt,fill] at (1,0) {};
\node[inner sep=0pt,circle,minimum size=4pt,fill] at (0,1) {};
\node[inner sep=0pt,circle,minimum size=4pt,fill] at (1,1) {};
\begin{scope}[shift={(1,2)}]
\draw[thick]   (-1,-1) .. controls (-1,-1) and (-1,1) .. (0,1) -- (3,1);
\draw  (0,1) -- (3,1) .. controls (3,-1) and (2,-1) .. (2,-1)  -- (-1,-1) .. controls (-1,-1) and (-1,1) .. (0,1);
\fill[pattern color = lightgray, pattern = north east lines] (0,1) -- (3,1) .. controls (3,-1) and (2,-1) .. (2,-1)  -- (-1,-1) .. controls (-1,-1) and (-1,1) .. (0,1);
\fill[color=white] (0,-.5) -- (2,-.5) -- (2, .5) -- (0, .5) -- (0,-.5);
\draw[dashed] (0,-.5) -- (2,-.5) -- (2, .5) -- (0, .5) -- (0,-.5);
\draw[ultra thick] (0,-1) .. controls (0,-0.5) and (1,-0.5).. (1,-0.5) -- (1,.5) .. controls (1,.5) and (2,0.5) .. (3,1);
\node at (1.3,0) {$k$}; 
\end{scope}

\node[inner sep=0pt,circle,minimum size=4pt,fill] at (2,1.5) {};
\node[inner sep=0pt,circle,minimum size=4pt,fill] at (2,2.5) {};


\begin{scope}[shift={(-1,0)}]
\draw[ultra thick]   (-1,-1) .. controls (-1,-1) and (-1,0) .. (0,0) -- (1,0);
\draw  (0,0) -- (3,0) .. controls (3,-1) and (2,-1) .. (2,-1)  -- (-1,-1) .. controls (-1,-1) and (-1,0) .. (0,0);
\fill[pattern color = lightgray, pattern = north east lines] (0,0) -- (3,0) .. controls (3,-1) and (2,-1) .. (2,-1)  -- (-1,-1) .. controls (-1,-1) and (-1,0) .. (0,0);
\draw[thick] (-1,-1) .. controls (0,0) and (2,-1).. (2,0);
\end{scope}

\draw[thick] (j)--(0,1);
\draw[ultra thick] (j)--(1,0);

\draw[ultra thick] (1,0)--(1,1);


\draw[dashed] (0,0)--(0,1)--(1,1)--(1,0)--(0,0);
\begin{scope}[shift={(2,0)}]
\draw[dashed,color=gray] (0,0)--(0,1)--(1,1)--(1,0)--(0,0);
\end{scope}

\node at (.3,.5) {$j$}; 
\node at (-1.5,.2) {$q$};
\node at (-1,-.8) {$r$};
\end{tikzpicture}
&
\begin{tikzpicture}[scale=.85]
\node[inner sep=0pt,circle,minimum size=4pt,fill] (j) at (0,0) {};
\node[inner sep=0pt,circle,minimum size=4pt,fill] at (1,0) {};
\node[inner sep=0pt,circle,minimum size=4pt,fill] at (0,1) {};
\node[inner sep=0pt,circle,minimum size=4pt,fill] at (1,1) {};
\node[inner sep=0pt,circle,minimum size=4pt,fill] at (2,1) {};
\node[inner sep=0pt,circle,minimum size=4pt,fill] at (3,0) {};
\node[inner sep=0pt,circle,minimum size=4pt,fill] at (3,1) {};
\node[inner sep=0pt,circle,minimum size=4pt,fill] at (4,1) {};

\node[inner sep=0pt,circle,minimum size=4pt,fill] at (2,0) {};
\node[inner sep=0pt,circle,minimum size=4pt,fill] at (2,2) {};
\node[inner sep=0pt,circle,minimum size=4pt,fill] at (5,2) {};
\node[inner sep=0pt,circle,minimum size=4pt,fill] at (4,2) {};
\node[inner sep=0pt,circle,minimum size=4pt,fill] at (3,2) {};


\begin{scope}[shift={(-1,0)}]
\draw[ultra thick]   (-1,-1) .. controls (-1,-1) and (-1,0) .. (0,0) -- (1,0);
\draw  (0,0) -- (3,0) .. controls (3,-1) and (2,-1) .. (2,-1)  -- (-1,-1) .. controls (-1,-1) and (-1,0) .. (0,0);
\fill[pattern color = lightgray, pattern = north east lines] (0,0) -- (3,0) .. controls (3,-1) and (2,-1) .. (2,-1)  -- (-1,-1) .. controls (-1,-1) and (-1,0) .. (0,0);
\draw[thick] (-1,-1) .. controls (0,0) and (2,-1).. (2,0);
\end{scope}

\draw[thick] (j)--(0,1);

\begin{scope}[shift={(2,0)}]
\draw[dashed,color=gray] (0,0)--(0,1)--(1,1)--(1,0)--(0,0);
\end{scope}
\begin{scope}[shift={(2,1)}]
\draw[dashed,color=gray] (0,0)--(0,1)--(1,1)--(1,0)--(0,0);
\end{scope}
\begin{scope}[shift={(3,0)}]
\draw[dashed,color=gray] (0,0)--(0,1)--(1,1)--(1,0)--(0,0);
\end{scope}
\begin{scope}[shift={(3,1)}]
\draw[dashed,color=gray] (0,0)--(0,1)--(1,1)--(1,0)--(0,0);
\end{scope}
\begin{scope}[shift={(4,1)}]
\draw[dashed,color=gray] (0,0)--(0,1)--(1,1)--(1,0)--(0,0);
\end{scope}

\draw[dashed] (0,0)--(0,1)--(1,1)--(1,0)--(0,0);
\begin{scope}[shift={(1,0)}]
\draw[dashed] (0,0)--(0,1)--(1,1)--(1,0)--(0,0);
\end{scope}

\node at (.3,.5) {$j_2$}; 
\node at (2.3,1.5) {$j_1$}; 

\node at (4.3,1.5) {$k'$}; 

\node at (3.3,.5) {$k$};
\draw [thick,dashed] (2,1) -- (0,1);
\draw [thick] (2,1)--(2,2);
\draw [thick,dashed] (2,2) --(5,2);

\draw [ultra thick] (0,0) -- (3,0) -- (3,1) -- (4,1) -- (4,2) -- (5,2);

\node at (-1.5,.2) {$q$};
\node at (-1,-.8) {$r$};
\draw (1,2) -- (4,-1);
\fill[pattern color=black,pattern = grid] (1,2) -- (4,-1) -- (3.8,-1) -- (.8,2);
\node at (2.5,-.5) {$<j_1$};
\node at (4.5,-.5) {$\ge j_1$};
\end{tikzpicture}
\end{tabular}
\end{center}
\label{fig:B}
\end{figure}

\begin{proof}
By Proposition~\ref{LPMprop:X}~{\em (i)}, we have
 $$\mathrm{rk}_M(\SET{1,2,\ldots,j-1})=\left|\left\{ i\in\SET{1,2,\ldots,j-1}\vphantom{A^A}~\middle|~ q_{i}=\mathrm{N}\right\} \right|.$$
Now fix some $k\in E$ with $k\geq j$. Since $M$ has no loop, 
 there is a base $B$ of $M$ with $k\in B$ and thus a lattice path $r=(r_{i})_{i=1}^{n}\in\mathrm{P}[p,q]$
with $r_{k}=\mathrm{N}$ (Theorem~\ref{LPMthm:P}).
We can construct a lattice path $s=(s_{i})_{i=1}^{n}\in\mathrm{P}[p,q]$
that follows $q$ for the first $j-1$ steps, then goes towards the
East until it meets $r$, and then goes on as $r$ does (Fig.~\ref{fig:B}a). 
The base $B_s = \SET{i\in E\mid s_i = \mathrm{N}}$
 that corresponds to the constructed path yields
\begin{align*}
\mathrm{rk}_M(\SET{1,2,\ldots,j-1}\cup\{k\}) & \geq \left|\left(\SET{1,2,\ldots,j-1}\vphantom{A^A}\cup\{k\}\right)\cap B_s\right|
\\ & = 1+\left|\left\{ i\in\SET{1,2,\ldots,j-1}\vphantom{A^A}~\middle|~ q_{i}=\mathrm{N}\right\} \right|
\\ & = 1+\mathrm{rk}_M(\SET{1,2,\ldots,j-1}).
\end{align*}
Since $\rk_M$ is unit increasing, 
adding a single element to a set can increase the rank by at most one,
 thus the inequality in the above formula is indeed an equality.
This implies that $k\notin\mathrm{cl}_M(\SET{1,2,\ldots, j-1})$. 
Since $k$ was arbitrarily chosen with $k \geq j$,
we obtain
 $\SET{1,2,\ldots, j-1}=\mathrm{cl}_M(\SET{1,2,\ldots, j-1})$.\end{proof}

\begin{theorem}\label{thm:WesternColine}
\label{LPMthm:B}Let $p=(p_{i})_{i=1}^{n}$, $q=(q_{i})_{i=1}^{n}$ be
lattice paths, such that $p\preceq q$ and such that $M=M[p,q]=(E,\Ical)$ 
has no loop and no parallel edges, and $\rk_M(E) \geq 2$.
Let $N_q = \SET{i\in E\mid q_i = \mathrm{N}}$, $j_{1}=\max N_q$, and $j_{2}=\max N_q\BSET{j_1}.$
Then the following holds
\begin{enumerate}\ROMANENUM
\item $\SET{1,2,\ldots, j_{2}-1}$ is a coline of $M$, we shall call it the \emph{Western coline of $ M$}.
\item $\SET{1,2,\ldots,j_{1}-1}$ is a copoint on the Western coline of $M$, which is a multiple copoint whenever $j_{1}-j_{2}\ge 2$. 
\item For every $k \geq j_1$ the set $\SET{1,2,\ldots,j_{2}-1}\cup\{k\}$ is a simple copoint on the Western coline of $M$.
\end{enumerate}
\end{theorem}

\begin{proof}
Lemma~\ref{LPMlem:A} provides that the set $W = \SET{1,2,\ldots,j_{2}-1}$ as well as the set $X = \SET{1,2,\ldots,j_{1}-1}$ 
is a flat of $M$. By construction of $j_1$ and $j_2$ we have that $\rk(W) = \rk(E) -2$ and $\rk(X) = \rk(E) - 1$.
Thus $W$ is a coline of $M$ --- so {\em (i)} holds --- and $X$ is a copoint of $M$,
which follows from  and the construction of $j_{2}$
and $j_{1}$. Since $\left| X\BS W \right| = \left| \SET{j_2, j_2+1,\dots, j_1-1} \right| = j_1 - j_2$
we obtain statement {\em (ii)}.
Let $k\geq j_1$, and let $X_k = \SET{1,2,\ldots,j_{2}-1}\cup\{k\}$.
 Lemma~\ref{LPMlem:A} yields that $\mathrm{rk}(X_k)=\mathrm{rk}(E)-1$, thus $\cl(X_k)$ is a copoint on the Western coline $W$.
 It remains to show that $\cl(X_k) = X_k$, which implies that $X_k$ is indeed a simple copoint on $W$.
 We prove this fact by showing that for all $k' \geq j_1$, $\rk(X_k\cup\SET{k'}) = \rk(E)$
  by constructing a lattice path. Without loss of generality we may assume that $k < k'$.
  Since $M$ has no loops and no parallel edges, there is a lattice path $r=(r_{i})_{i=1}^{n}\in\mathrm{P}[p,q]$
  with $r_k = r_{k'} = \mathrm{N}$.
There is a lattice path $s=(s_{i})_{i=1}^{n}\in\mathrm{P}[p,q]$  that follows $q$ for
the first $j_{2}-1$ steps, then goes towards the East until it meets
$r$, and then goes on as $r$ does (Fig.~\ref{fig:B}b). The constructed
path $s$ yields that 
\begin{align*}
  \mathrm{rk}(X_k\cup\{k'\}) & \geq\left|\left(W\cup\{k,k'\}\right)\cap \SET{i\in E\mid s_i=\mathrm{N}} \right|
  \\ & =2+\left|W\cap \SET{i\in E\mid q_i =\mathrm{N}}\right|
\\& 
  =2+\mathrm{rk}(W) = 1 + \rk(X_k) = 1+ \rk(X_{k'}),
\end{align*}
where $X_k' = W\cup\SET{k'}$. Thus $k'\notin \cl(X_k)$ and $k\notin \cl(X_k')$. This
 completes the proof of statement {\em (iii)}.
\end{proof}

\begin{theorem}\label{thm:simplelpmqsc}
Let $M=(E,\Ical)$ be a strong lattice path matroid with $\rk_M(E) \geq 2$ such that 
$\left| E \right| = n$ and such that $M$ has neither a loop nor a pair of parallel edges.
Then either the Western
coline is quite simple, or the element $n\in E$ is a coloop, and in the latter case
there is either another
coloop or $\rk_M(E) \ge3$.\end{theorem}

\begin{proof}
If $j_{1}\leq n-1$ as defined in Theorem~\ref{LPMthm:B}, $W=\SET{1,2,\ldots,j_{2}-1}$
has at most a single multiple copoint and at least two simple copoints, therefore
it is quite simple. Otherwise $j_{1}=n$ is a coloop. If there is another
coloop $e_{1}$, then $\SET{1,2,\ldots,n-1}\backslash\{e_{1}\}$ is a quite simple coline
with two simple copoints. If $n$ is the only coloop, the rank of $M$ is $2$, and
there is no other coloop, then this would imply that there are parallel edges --- a contradiction to the assumption that $M$ is a simple matroid.
\end{proof}

\section{Lattice Path Matroids are 3-Colorable}

\begin{corollary}\label{cor:LPMqscoline}
Every simple lattice path matroid $M=(E,\Ical)$ with $\rk_M(E)\geq 2$ has a quite simple
coline.\end{corollary}
\begin{proof}
Without loss of generality, we may assume that $M$ is a strong lattice path matroid on $E=\SET{1,2,\ldots,n}$,
and we may use $j_1$ and $j_2$ as defined in Theorem~\ref{LPMthm:B}.
From Theorem~\ref{thm:simplelpmqsc}, we obtain the following:
If $j_1 < n$, the Western coline is quite simple. Otherwise, if $j_1=n$, then $n$ is a coloop.
If there is another coloop $e_{1}$, then $\SET{1,2,\ldots,n-1}\backslash\{e_{1}\}$ is a quite simple coline.
If there is no other coloop, then we have $\rk_M(E) \geq 3$,
 and the contraction $M' = M\contract E\BSET{n}$ is a strong lattice path matroid
without loops, without parallel edges, and
without coloops, such that $\rk_{M'}(E\BSET{n})=\rk_M(E) - 1 \geq 2$. Thus the corresponding $j_1' < n -1$ 
and the Western coline $W'$ of $M'$ is quite simple in $M'$ (Theorem~\ref{thm:simplelpmqsc}).
But then $\tilde W = W'\cup\SET{n}$ is a coline of $M$, and $\tilde X$ is a copoint on $\tilde W$ with respect to $M$
if and only if $X' = \tilde X \BSET{n}$ is a copoint on $W'$ with respect to $M'$. 
Since $\left| \tilde W \BS \tilde X \right| = \left| W' \BS X' \right|$, we obtain that $\tilde W$ is a quite simple coline of $M$.
\end{proof}

\begin{definition}[\cite{GoHoNe15}, Definition~2] Let $\mathcal{O}$ be an oriented
matroid. We say that $\mathcal{O}$ is \emph{generalized series-parallel},
if every non-trivial minor $\Ocal'$ of $\mathcal{O}$ with a simple underlying matroid $M(\Ocal')$ has a $\{0,\pm1\}$-valued
coflow which has exactly one or two nonzero-entries.\end{definition}

\begin{lemma}[\cite{GoHoNe15}, Lemma~5]\label{lem:quiteSimpleColineQGSP}
If an orientable matroid $M$ has a quite simple
coline, then every orientation $\mathcal{O}$ of $M$ has a $\{0,\pm1\}$-valued
coflow which has exactly one or two nonzero-entries.
\end{lemma}

 For a proof, see \cite{GoHoNe15}.

\begin{remark}\label{rem:GPSloRank}
  A simple matroid of rank $1$ has only one element, no circuit and a single cocircuit consisting of the sole element of the matroid;
  so every rank-$1$ oriented matroid is generalized series-parallel.
  Observe that every simple matroid $M=(E,\Ical)$ with $\rk_M(E) = 2$ is a lattice path matroid,
  as it is isomorphic to the
  strong lattice path matroid $M[p,q]$ where $p=(p_i)_{i=1}^{\left| E \right|}$ 
  with \[ p_i = \begin{cases} \mathrm{E} &\quad \text{if~} i < \left| E \right| - 2, \\
                              \mathrm{N} &\quad \text{otherwise,} \end{cases}\]
  and where $q=(q_i)_{i=1}^{\left| E \right|}$ 
  with \[ q_i = \begin{cases} \mathrm{N} &\quad \text{if~} i \leq 2,\\
                              \mathrm{E} &\quad \text{otherwise.} \end{cases}\]
  Therefore Lemma~\ref{lem:quiteSimpleColineQGSP} and Corollary~\ref{cor:LPMqscoline} yield that $\Ocal$ has a $\{0,\pm1$\}-valued
coflow which has exactly one or two nonzero-entries.
Consequently, every oriented matroid $\Ocal =(E,\Ccal,\Ccal^\ast)$ with
  $\rk_{M(\Ocal)}(E) \leq 2$ is generalized series-parallel.    
\end{remark}


\begin{corollary}\label{cor:LPMGSP}
All orientations of lattice path matroids are generalized series-parallel.
\end{corollary}
\begin{proof}
Lemma~\ref{lem:quiteSimpleColineQGSP}, Remark~\ref{rem:GPSloRank}, Theorem~\ref{thm:LPMclosedUnderStuff} and Corollary~\ref{cor:LPMqscoline}.
\end{proof}

\begin{theorem}[\cite{GoHoNe15}, Theorem 3]
\label{thm:GPS3col} Let $\Ocal=(E,\Ccal,\Ccal^\ast)$ be a generalized series-parallel oriented
matroid such that $M(\Ocal)$ has no loops.
Then there is a nowhere-zero coflow $F\in \Z.\Ccal^\ast$ such that $\left| F(e) \right| < 3$ for all $e\in E$.
Thus $\chi(\Ocal) \leq 3$.
\end{theorem}

 For a proof, see \cite{GoHoNe15}.

\begin{corollary}
  Let $\Ocal$ be an oriented matroid such that $M(\Ocal)$ is a lattice path matroid without loops.
  Then \( \chi(\Ocal) \leq 3 \).
\end{corollary}
\begin{proof}
Theorem~\ref{thm:GPS3col} and Corollary~\ref{cor:LPMGSP}.
\end{proof}


\bigskip
\footnotesize
\noindent\textit{Acknowledgments.}
This research was partly supported by a scholarship granted by the FernUniversität in Hagen.


\bibliographystyle{plain}

\bibliography{references.bib}

\end{document}